\documentclass[a4paper,12pt,english]{article}
\usepackage[utf8x]{inputenc}
\usepackage{amsmath, amsthm, amssymb, graphicx, makeidx, mathrsfs, enumerate}
\usepackage{setspace}
\usepackage{hyperref}
\usepackage{multirow}
\usepackage[T1]{fontenc}
\usepackage{babel}
\usepackage{authblk}

\newcommand{\Z}{\mbox{$\mathbb Z$}}	
\newcommand{\Q}{\mbox{$\mathbb Q$}}	
\newcommand{\N}{\mbox{$\mathbb N$}}     
\newcommand{\R}{\mbox{$\mathbb R$}}     
\newcommand{\C}{\mbox{$\mathbb C$}}     

\usepackage{vmargin}
\setmarginsrb     { 1.25in}  
                        { 1.0in}  
                        { 1.0in}  
                        { 1.0in}  
                        {  20pt}  
                        {0.25in}  
                        {   9pt}  
                        { 0.3in}  
\newtheorem{theorem}{Theorem}[section]

\newtheorem{lemma}[theorem]{Lemma}

\makeindex
\begin{document}
\pagenumbering{arabic}
\author[]{Anuj Jakhar\footnote{All correspondence may be addressed to this author.} }

\author[]{Neeraj Sangwan \\}
\author[]{anujjakhar@iisermohali.ac.in,  neerajsan@iisermohali.ac.in}
\affil[]{ Indian Institute of Science Education and Research (IISER), Mohali
 Sector-81, S. A. S. Nagar-140306, Punjab, India.}

\date{}
\renewcommand\Authands{}

\title{ An irreducibility criterion for integer polynomials\footnote{\noindent The financial support from IISER Mohali is gratefully acknowledged by the authors. }}

\maketitle

\noindent{\textbf {Abstract.} Let $f(x) = \sum\limits _{i=0}^{n} a_i x^i $ be a polynomial with coefficients from the ring $\Z$ of integers satisfying either $(i)$  $0 < a_0 \leq a_{1} \leq \cdots \leq a_{k-1} < a_{k} < a_{k+1} \leq \cdots  \leq a_n$ for some $k$, $0 \leq k \leq n-1$; or
$(ii)$ $|a_n| > |a_{n-1}| + \cdots + |a_{0}|$ with $a_0 \neq 0$. In this paper, it is proved that if $|a_n|$ or $|f(m)|$ is a prime number for some integer $m$ with $|m|\geq 2 $ then the polynomial $f(x)$ is irreducible over $\Z$.

\bigskip

\noindent \textbf{Keywords :} Irreducible polynomials, Integer polynomials.

\bigskip

\noindent \textbf{2010 Mathematics Subject Classification }: 12E05.
\newpage

\section{Introduction}
One of the major themes in the development of number theory is the relationship between prime numbers and irreducible polynomials.
  There are better methods to determine prime numbers than to determine irreducible polynomials over the ring $\Z$ of 
  integers. In 1874, Bouniakowsky \cite{BC} made a conjecture that if $f(x) \in \Z[x]$ is an irreducible polynomial for 
  which the set of values $f(\Z^{+})$ has greatest common divisor $1$, then $f(x)$ represents prime numbers infinitely 
  often. Bouniakowsky conjecture is true for polynomials of degree $1$ in view of the well known Dirichlet's Theorem for 
  primes in arithmetic progression's; however it is still open for higher degree polynomials. The converse of the 
  Bouniakowsky conjecture, viz. if $f(x)$ is an integer polynomial such that the set $f(\Z^+)$ has infinitely many prime 
  numbers, then $f(x)$ is irreducible over $\Z$, is true; because if $f(x)=g(x)h(x)$ where $g(x)$, $h(x) \in \Z[x]$, then at least one 
  of the polynomials $g(x)$, $h(x)$, say $g(x)$ takes the values $\pm 1$  infinitely many times which is not possible, 
  since the polynomials $g(x) + 1$ and $g(x) - 1$ can have at most finitely many roots over the field of complex numbers.  
  In this paper, we give a proof (using elementary methods) of a similar version of the converse for a special class of 
  polynomials with a weaker hypothesis. Indeed we prove the following theorem:
\begin{theorem} \label{1.1} Let $f(x) = a_n x^n +a_{n-1}x^{n-1} + \cdots+ a_1x + a_0 \in  \Z[x]$ be a polynomial which satisfies one of the following conditions: \\
$(i)$ $0 < a_0 \leq a_{1} \leq \cdots \leq a_{k-1} < a_{k} < a_{k+1} \leq \cdots  \leq a_n$ for some $k$, $0 \leq k \leq n-1$;\\
$(ii)$ $|a_n| > |a_{n-1}| + \cdots + |a_{0}|$ with $a_0 \neq 0$.\\
Suppose that $|a_n|$ is a prime number or $|f(m)|$ is a prime for some integer $m$ with $|m| \geq 2$.  Then $f(x)$ is irreducible in $\Z[x]$. Further if $|m|$ is the $r^{th}$ power of some integer, then $f(x^{r})$ is irreducible in $\Z[x]$. 
\end{theorem}
It may be pointed out that a similar result is proved in \cite[Theorem 1]{RM} with a different hypothesis.
\section{Preliminary Results}
Let the set $ \left\lbrace z \in \C : |z|<1 \right\rbrace  $ be denoted by $\mathcal{C} $, where $\C$ denotes the complex numbers.
The following proposition proved in \cite[Proposition 2.3]{HJ} will be used in the sequel. We omit its proof.\\

\noindent\textbf{Proposition 2.A.}
Let $f(x) = \sum\limits_{l = 0}^{n}a_l x^l \in \Q[x]$ and suppose that $a_i \neq 0$ and $a_{j} \neq 0$ for some $0 \leq i < j \leq n$. Suppose further that \begin{equation}\label{eq:irr1}
\sum\limits_{0 \leq l \leq n; l \neq t}|a_{l}| \leq q^t|a_t| 
\end{equation}      
for some $0 \leq t \leq n$, with $t \neq i$ and $t \neq j$, and some $q \in \R$ with $0 < q \leq 1$. If $f(x)$ has a zero $\alpha \in \{z \in \C | q \leq |z| \leq 1\}$, then equality holds in $(\ref{eq:irr1})$ and $\alpha^{2(j-i)} = 1$.\\

\noindent Now we prove  two elementary lemmas which are of independent interest as well.

\begin{lemma}
Let $f(x) = a_n x^n +\cdots+a_0 \in  \Z[x]$ be a polynomial such that $0 < a_0 \leq a_{1} \leq \cdots \leq a_{k-1} < a_{k} < a_{k+1} \leq \cdots \leq a_{n-1} \leq a_n$ for some $k$, $0 \leq k \leq n-1$. Then $f(x)$ has all zeros in the set $\mathcal{C}$.
\end{lemma}
\begin{proof}
We first show that $f(x)$ has all zeros in $\{z \in \C : |z| \leq 1\}$. Suppose to the contrary that $f(x)$ has a zero $\alpha$ with $|\alpha|>1$. Then $\alpha$ is a root of $F(x) = (x-1)f(x) = a_nx^{n+1} + (a_{n-1} - a_n)x^{n} + \cdots + (a_{0} - a_{1})x - a_0.$ Therefore in view of the hypothesis and the assumption $|\alpha| > 1$, we have 
\begin{align*}
|a_n\alpha^{n+1}|&\leq a_0 + (a_1 - a_0)|\alpha| + \cdots + (a_n - a_{n-1})|\alpha|^{n}\\
& < a_0|\alpha|^n + (a_1 - a_0)|\alpha|^{n} + \cdots + (a_n - a_{n-1})|\alpha|^{n} = |a_n\alpha^n|
\end{align*}
which is a contradiction as $a_n > 0$.  Now we show that $|\alpha| < 1$. Assume that $|\alpha|=1$.
Observe that the coefficients of $x^{k}$ and $x^{k+1}$ in $F(x)$ are negative and other coefficients except $a_n$ are non-positive. Thus the hypothesis of Proposition 2.A is satisfied for $t = n+1, i = k, j = k+1$ and $q = 1$. By this proposition, we have $\alpha^2 = 1$, which is impossible as $f(1)$ and $f(-1)$ are easily seen to be non-zero using the hypothesis.
\end{proof}

\begin{lemma}
Let $f(x) \in \Z[x]$ be a polynomial having all its zeros in the set $\mathcal{C}$. If there exists an integer $m$ with $|m| \geq 2$ such that  $|f(m)|$ is a prime number, then $f(x)$ is irreducible in $\Z[x]$.
\end{lemma}
\begin{proof}
 Suppose to the contrary that $f(x) = g(x)h(x)$, where $g(x),h(x)\in \Z[x]$. In view of the hypothesis, at least one of $|g(m)|, |h(m)|$ equals to 1. Say $|g(m)| = 1$. Write $g(x) = c\prod\limits_{i=1}^{k}(x - \alpha_i)$, where $\alpha_{i} \in \C,~1\leq i \leq k$. Keeping in mind that $|\alpha_i| < 1$ and $|m| \geq 2$, we have
\begin{align*}
|g(m)| &= |c\prod\limits_{i=1}^{k}(m - \alpha_i)| \geq |c|\prod\limits_{i=1}^{k}(|m| - |\alpha_i|) > |c|\prod\limits_{i=1}^{k}(|m| - 1) \geq 1,
\end{align*} 
which gives $|g(m)| > 1$, leading to a contradiction. Hence the lemma is proved.
\end{proof}
Arguing as in the above lemma, the following lemma can be easily proved.
\begin{lemma}
Let $f(x) \in \Z[x]$ be a polynomial having all its zeros in the set $\{z \in \C : |z| > 1\}$. If $|f(0)|$ is  a prime number, then $f(x)$ is irreducible in $\Z[x]$.
\end{lemma}

\section{Proof of the Theorem 1.1.}
 Note that if $f(x)$ satisfies condition $(i)$ of the theorem, then in view of Lemma 2.1, it has all its zeros  inside the unit circle $\mathcal{C} = \left\lbrace z \in \C : |z|<1 \right\rbrace  $. Observe that if $(ii)$ holds and $|\alpha| \geq 1$, then 
 $$|f(\alpha)| \geq |a_n||\alpha|^n - \sum\limits_{j=0}^{n-1}|a_j||\alpha|^j \geq |\alpha|^n \left(|a_n| - \sum\limits_{j=0}^{n-1}|a_j|\right) > 1,$$ so $f(\alpha) \neq 0$. Thus, in case  $f(x)$ satisfies condition $(ii)$ of the theorem then all its roots lie in $\mathcal{C}$. As $a_0 \neq 0$ by hypothesis, it now follows that all roots of the polynomial $x^nf(\frac{1}{x})=g(x)$(say) lie in the set $ \left\lbrace z\in \C : |z|>1 \right\rbrace $.  Therefore when $|a_n| =|g(0)| $ is a prime number then $g(x)$ is irreducible over $\Z$ by virtue of Lemma 2.3 and hence is $f(x)$. If there exists an integer $m$ with $|m|\geq 2$ such that $|f(m)|$ is prime then $f(x)$ is irreducible over $\Z$ in view of Lemma 2.2. It only remains to be shown that if $|m|$ is a $r^{th}$ power of some integer  then $f(x^r)$ is irreducible over $\Z$. Since all the roots of $f(x) $ lie in $\mathcal{C}$, therefore so do the roots of $f(\pm x^r) $. As $|f(m)|$ is a prime number, the desired assertion follows from Lemma 2.2. 

 \medskip
 The following examples are quick applications of Theorem 1.1.
 \medskip\\
 \noindent {\bf Example 3.1.}
Let $2^{2^n}+1$ be a prime number for some $n\in \N$. Then the polynomial $F(x)= (2^{2^n}+1)x^n +a_{n-1}x^{n-1}+\cdots  +a_1 x +a_0$ with $1\leq a_0\leq \cdots \leq a_{k-1}<a_k < a_{k+1}\leq \cdots \leq a_{n-1}\leq 2^{2^n}+1$ for some $1\leq k\leq n-1$ is irreducible by Theorem $\ref{1.1}$. \\
 
 \noindent {\bf Example 3.2.} Let $F(x)= 10x^5 +3x^4 +2x^3+x^2 +1$ be a polynomial. Since $F(2)=389$ is a prime, $F(x)$ is irreducible in view of Theorem \ref{1.1}.\\
 
 \noindent {\bf Example 3.3.} Let $f(x) = 2^{51}x^{10} + 2^{10}x^9 - 2^{11}x^{8} + 2^{15}x^{7} - 2^{16}x^{6} - 1$ be a polynomial. One can check that $f(2) = 2^{61} - 1$ is a prime number (Mersenne prime) and $f(x)$ satisfies condition $(ii)$ of the Theorem 1.1. Therefore $f(x)$ is irreducible in $\Z[x]$ in view of Theorem 1.1.

\end{document}